\documentclass[11pt]{article}

\usepackage{amsmath,amsfonts,amsthm,amssymb, mathtools}
\usepackage{mathrsfs,multirow,graphicx,color,latexsym, tikz, calc}

\usetikzlibrary{shadows,patterns,arrows,decorations.pathreplacing}

\newtheorem{thm}{Theorem}[section]
\newtheorem{lem}[thm]{Lemma}
\newtheorem{cor}[thm]{Corollary}

\newtheorem{prop}[thm]{Proposition}

\newtheorem{examplex}{Example}
\newenvironment{example}
{\pushQED{\qed}\examplex}
{\popQED\endexamplex}

\title{\bf \Large A graph discretization of vector Laplace operator}

\author{
{\small  Shu Li$^a$,  \ \ Lu Lu$^b$, \ \ Jianfeng Wang$^{a,}$\footnote{Corresponding author.\newline{\hspace*{5mm}Email addresses: shuligraph@gmail.com (S. Li), lulugdmath@163.com (L. Lu),  jfwang@sdut.edu.cn (J.F.Wang).}}}\\[2mm]
\footnotesize $^a$School of Mathematics and Statistics, Shandong University of Technology, Zibo 255049, China\\
\footnotesize $^b$School of Mathematics and Statistics, HNP-LAMA, Central South University,  Changsha 410083, China
}

\date{}

\begin{document}
\maketitle
	
\setcounter{page}{1}
\begin{abstract}
In this paper, we study the graph-theoretic analogues of vector Laplacian (or Helmholtz operator) and   vector Laplace equation. We  determine the graph matrix representation of vector Laplacian and obtain the dimension of solution space of vector Laplace equation on graphs.\\

\end{abstract}


\section{Introduction}

In Mathematics and Physics, Laplace's equation is one of the most famous differential equation
named after Pierre-Simon de Laplace, who firstly presented and investigated its properties. In general, this is often written as $$\nabla^2 f = 0,$$
where $\nabla^2$ is the  Laplace operator (or Laplacian) operating a scalar function $f$. With that in mind, the Laplacian is also referred to as {\it scalar Laplacian}. In the context of real space $\mathbb{R}^3$ and rectangular coordinates, then $$\nabla^2 f = \operatorname{div}\operatorname{grad}f = \frac{\partial^2 f}{\partial x_1^2} + \frac{\partial^2 f}{\partial x_2^2} + \frac{\partial^2 f}{\partial x_3^2},$$
where $\operatorname{div} = \nabla \cdot$ and $\operatorname{grad} = \nabla$ are respectively the {\it divergence} and {\it gradient} operators. Recall, $$\operatorname{grad}f=\nabla f=(\frac{\partial f}{\partial x_1},\frac{\partial f}{\partial x_2},\frac{\partial f}{\partial x_3}) \quad \mbox{and} \quad \operatorname{div} \mathbf{F}=\nabla\cdot \mathbf{F}=\frac{\partial F_1}{\partial x_1}+\frac{\partial F_2}{\partial x_2}+\frac{\partial F_3}{\partial x_3},$$ where $\mathbf{F}$ is a vector field.

In the setting of Graph Theory, the scalar Laplacian $\nabla^2 = -{\rm div\!~grad}$ gives the celebrated {\it Laplacian matrix} of a graph \cite[Lemma 5.6, eg.]{lim2020}, defined by $$L(G) = D(G)-A(G)$$ where $D(G) = {\rm diag}(d(v_1), d(v_2),\cdots, d(v_n))$ is the degree diagonal matrix with $d(v)$ being the degree of vertex $v$  and  $A(G)=(a_{ij})$ is the  adjacency matrix of $G$ in which $a_{ij}=1$ if $v_iv_j \in E(G)$ and $0$ otherwise. Proverbially, the graph Laplacian has been studied extensively in Spectral Graph Theory  and has so much influence on many areas. For good survey articles on the graph Laplacian the reader is referred to \cite{mer,moh}.

We here pay attention to the {\it vector Laplacian}. In $\mathbb{R}^3$, the {\it curl} of a vector field  $\mathbf{F}$ is defined to be $$\operatorname{curl} \mathbf{F}=\nabla\times \mathbf{F}=(\frac{\partial F_3}{\partial x_2}-\frac{\partial F_2}{\partial x_3},\frac{\partial F_1}{\partial x_3}-\frac{\partial F_3}{\partial x_1},\frac{\partial F_2}{\partial x_1}-\frac{\partial F_1}{\partial x_2}).$$ Then the {\it vector Laplace equation} is defined as
\begin{equation}
\nabla^2 \mathbf{F}= (\nabla^2 \mathbf{F}_1,\nabla^2 \mathbf{F}_2,\nabla^2 \mathbf{F}_3) =0.
\end{equation}
A  straightforward calculation shows that
\begin{equation}\label{vector-F}
\nabla^2 \mathbf{F}= \operatorname{grad}\operatorname{div}\mathbf{F}-\operatorname{curl}\operatorname{curl} \mathbf{F},
\end{equation}
which is called {\it vector Laplace operator} (or {\it vector Laplacian} \cite{moon1953}) or {\it Helmholtz operator} \cite{jiang-lim};.

What is clear from the literature is that the operators involving $\nabla^2 = \nabla\cdot\nabla$ on a vector function start back in 1911
\cite[p133]{coffin1911}. However, the difference between the scalar and vector Laplacians was not unrecognized in a long period, which caused that the progress
in this direction had been hampered. The meaning of the vector Laplacian was not clear untill 1953, credit to Moon and Spencer \cite{moon1953}, who developed a
general equation for the vector Laplacian in any orthogonal, curvilinear coordinate system. Remark that the scalar and vector Laplacians are special cases of
Hodge Laplacians, named from the Hodge theory on graphs \cite{lim2020}, and that the Laplace equation and vector Laplace equation are respectively the special
cases of Helmholtz equations $$\nabla^2 \theta - k^2\theta=0,$$ in which $\theta$ is a scalar function or a vector field.

For the purpose here, we consider the adjoint operators of $\operatorname{grad}$ and $\operatorname{curl}$ denoted by $\operatorname{grad}^*$ and $\operatorname{curl}^*$ respectively. Due to $\operatorname{div} = -\operatorname{grad}^*$ and $\operatorname{curl} = \operatorname{curl}^*$ \cite[pp. 207]{stra-book}, then an equivalent expression of \eqref{vector-F} is shown as follows
\begin{equation}\label{vector-F1}
-\nabla^2 \mathbf{F}= \operatorname{grad}\operatorname{grad}^*\mathbf{F}+\operatorname{curl}^*\operatorname{curl} \mathbf{F}.
\end{equation}

In this paper, we fucus on the graph-theoretic analogues of vector Laplacian and vector Laplace equation. In Section 2 we determine the graph matrix representation, {\it Helmholzian matrix}, of vector Laplacian, which put the way for establish a spectral graph theory based on this new matrix. In Section 3 we identify the dimension of solution space of vector Laplace equation on graphs. As a corollary,  the number of triangles in a graph is also obtained. In Section 4 we give some remarks and point out some potential applications of the results obtained in this paper.

\section{Vector Laplacian on Graphs}

Let $G =(V(G),E(G))$ be an undirected simple graph with vertex set $V(G) =\{v_1,v_2,\cdots,v_n\}$ and edge set $E(G) = \{e_1,e_2,\cdots,e_m\}$, where its order is $n=|V(G)|$ and its size is $m=|E(G)|$.  Let $T$ be the set of triangles in $G$. We define the real valued functions on its vertex set $\phi: V \rightarrow R$. Moreover,  we request the real valued functions on $E$ and $T$ to be alternating. By an alternating function on $E$, we mean a function of the form $\varphi: V \times V \rightarrow \mathbb{R}$, where
$$
\varphi(i,j)=
\begin{cases}
-\varphi(j,i), & \mbox{if $\{v_i,v_j\}\in E$};\\
0, & \mbox{otherwise}.
\end{cases}
$$
An alternating function on $T$ is one of the form $\psi: V \times V \times V \rightarrow \mathbb{R}$, where
$$
\varphi(i,j,k)=\psi(j,k,i)=\psi(k,i,j)=
\begin{cases}
-\psi(j,i,k)=-\psi(i,k,j)=-\psi(k,j,i), & \mbox{if $\{v_i,v_j,v_k\}\in T$};\\
0, & \mbox{otherwise}.
\end{cases}
$$

From the topological point of view, the functions $\phi,\varphi$ and $\psi$ are called 0-, 1-, 2-{\it cochains}, which are discrete analogues of differential forms on manifolds \cite{warner}. Let the Hilbert spaces of $0$-, $1$- and $2$-cochains be respectively $L^2(V)$, $L^2_{\wedge}(E)$ and $L^2_{\wedge}(T)$ with $\wedge$ indicating alternating and with standard $L^2$-inner products defined by
$$\langle \phi_1,\phi_2\rangle_V=\sum_{i\in V}\phi_1(i)\phi_2(i),\;\langle \varphi_1,\varphi_2\rangle_E=\sum_{i\le j}\varphi_1(i,j)\varphi_2(i,j),\;
\langle \psi_1,\psi_2\rangle_T=\sum_{i<j<k}\psi_1(i,j,k)\psi_2(i,j,k).$$

We describe the graph-theoretic analogues of $\operatorname{grad}$, $\operatorname{curl}$, and $\operatorname{div}$ in multivariate calculus \cite{lim2020}. The {\it gradient} is the linear operator grad: $L^2(V) \rightarrow L_{\wedge}^2(E)$ defined by
$$
(\operatorname{grad}\phi)(i,j)=\phi(j)-\phi(i)
$$
for all $\{i, j\} \in E$ and zero otherwise. The {\it curl} is the linear operator curl: $L_{\wedge}^2(E) \rightarrow L_{\wedge}^2(T)$ defined by
$$
(\operatorname{curl}\varphi)(i,j,k)=\varphi(i,j)+\varphi(j,k)+\varphi(k,i)
$$
for all $\{i,j,k\} \in T$ and zero otherwise. The {\it divergence} is the linear operator div: $L_{\wedge}^2(E) \rightarrow L^2(V)$ defined by
$$
(\operatorname{div}\varphi)(i)=\sum_{j=1}^n\varphi(i,j)
$$
for all $i\in V$. Then the graph-theoretic analogue of the vector Laplacian, called {\it graph Helmholtzian},  is defined by $\Lambda= -\operatorname{grad}\operatorname{div}+\operatorname{curl^{\ast}}\operatorname{curl}$ \cite[pp. 692]{lim2020}. From $\operatorname{div}=-\operatorname{grad^{\ast}}$ \cite[Lemma 5.4]{lim2020} it follows that $\Lambda$ can be expressed as
\begin{equation}\label{eqhelm}
\Lambda = \operatorname{grad}\operatorname{grad^{\ast}}+\operatorname{curl^{\ast}}\operatorname{curl}.
\end{equation}
Note, Helmholtz Decomposition Theorem for the clique complex of a graph correlates closely with the kernel of graph Helmholtzian $\Lambda$ \cite[Theorem 2]{jiang-lim}. In addition, the graph Helmholtzian is a special case of {\it Hodge Laplacians}, a higher-order generalization of the graph Laplacian, due to Lim \cite{lim2020}.

\begin{table}[htbp]
\centering
\begin{tabular}{|c|c|c|c|}
\hline
\multicolumn{3}{|c|}{Symbol} & Diagram \\
\hline
\multirow{2}*{$u\sim v$} & \multicolumn{2}{c|}{$u\rightarrow v$} &
\includegraphics[scale=0.6]{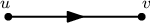}  \\[1.6mm]
\cline{2-4}
~ & \multicolumn{2}{c|}{$v\rightarrow u$} & \includegraphics[scale=0.6]{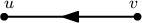}  \\[1.6mm]
\hline
\multirow{2}*{$u\in e$} & \multicolumn{2}{c|}{$u\rightarrow e$~($u=e^-$)} &
\includegraphics[scale=0.6]{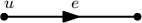}  \\[1.6mm]
\cline{2-4}
~ & \multicolumn{2}{c|}{$e\rightarrow u$~($u=e^+$)} & \includegraphics[scale=0.6]{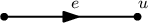}  \\[1.6mm]
\hline
\multirow{4}*{$e_1\sim e_2$} & \multirow{2}*{$e_1\overset{\pm}{\sim}e_2$} &
$e_1\overset{+}{\sim}e_2$ & \includegraphics[scale=0.6]{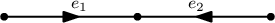}  \\[1.6mm]
\cline{3-4}
~ & ~ & $e_1\overset{-}{\sim}e_2$ & \includegraphics[scale=0.6]{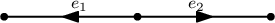}  \\[1.6mm]
\cline{2-4}
~& \multirow{2}*{$e_1\leftrightarrow e_2$} & $e_1\rightarrow e_2$ &
\includegraphics[scale=0.6]{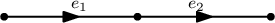}\\[1.6mm]
\cline{3-4}
~ & ~ & $e_2\rightarrow e_1$ &  \includegraphics[scale=0.6]{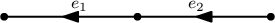}\\[1.6mm]
\hline
\multicolumn{3}{|c|}{$e_1\vartriangle e_2$} &
\raisebox{-.5\height}{\includegraphics[scale=0.6]{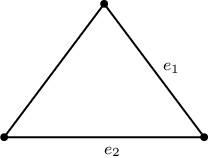}}\\[1.6mm]
\hline
\multirow{2}*{$e\in\vartriangle$} & \multicolumn{2}{c|}{$e\in\vartriangle^+$} &
\raisebox{-.5\height}{ \includegraphics[scale=0.6]{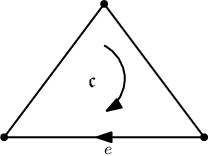}}\\[1.6mm]
\cline{2-4}
~ & \multicolumn{2}{c|}{$e\in\vartriangle^-$} &
\raisebox{-.5\height}{\includegraphics[scale=0.6]{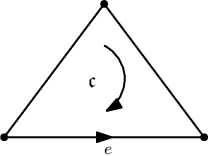}} \\[1.6mm]
\hline
\end{tabular}
\caption{The relations between vertices, edges and triangles.}
\label{tab-1}
\end{table}

We emphatically derive a graph matrix representation of graph Helmholtzian. Given arbitrary orientations to the edges and triangles of $G$, the {\it trail}
and the {\it head} of an oriented edge $e$ are respectively marked by $e^-$ and $e^+$. Set
$\mathcal{V}(e)=\{e^-,e^+\}$. If there is a directed edge from $u$ to $v$, then we write
$u\rightarrow v$. If two vertices $u$ and $v$ are adjacent, then we write $u\sim v$, and $u
\nsim v$ otherwise.  Therefore, $u\sim v$
implies either $u\rightarrow v$ or $v\rightarrow u$. If a vertex $v$ satisfies $v\in
\mathcal{V}(e)$, then we write $v \in e$. Furthermore, set $v\rightarrow e$ if $v=e^-$, and
$e\rightarrow v$ if $v=e^+$. For two edges $e_1$ and $e_2$, let $e_1\sim e_2$ if
$\mathcal{V}(e_1)\cap
\mathcal{V}(e_2)\ne\emptyset$. Put
$e_1\rightarrow e_2$ if $e_1^+=e_2^-$, $e_1\overset{+}{\sim}e_2$ if $e_1^+=e_2^+$, and
$e_1\overset{-}{\sim}e_2$  if $e_1^-=e_2^-$. Denote by $e_1\leftrightarrow e_2$ if
either $e_1\rightarrow e_2$ or $e_2 \rightarrow e_1$, and denote by $e_1\overset{\pm}{\sim}e_2$
if either $e_1\overset{+}{\sim}e_2$ or $e_1\overset{-}{\sim}e_2$. Set $e_1\vartriangle e_2$ if $e_1$
and $e_2$ are in a same triangle. For an edge $e$ and a triangle
$\vartriangle$, write $e\in\vartriangle$ if $e$ is an edge of $\vartriangle$. Furthermore, if
the orientation of $e$ is coincident with that of $\vartriangle$ then we write
$e\in\vartriangle^+$, and $e\in\vartriangle^-$ otherwise.  To make the symbols more clear, we
collect them in Tab. \ref{tab-1}.

For an edge $e \in E(G)$, the {\it triangle degree of $e$}, denoted by $\triangle_G(e),$ is the number of triangles containing $e$, that is,
$$\triangle_G(e)=|\{\vartriangle\in T(G)\mid e\in \vartriangle\}|.$$   The {\it edge-vertex incidence matrix}
$\mathcal{B}(G)=(b_{ev})_{m\times n}$ and {\it triangle-edge incidence matrix}
$\mathcal{C}(G)=(c_{\vartriangle e})_{t\times m}$ are severally defined by
\begin{equation}\label{B-C}
b_{ev}=
\left\{\begin{array}{cc}
-1,& v\rightarrow e\\
1,&e\rightarrow v\\
0,&\textrm{otherwise}
\end{array}\right. \;\;\textrm{and}\;\; c_{\vartriangle e}=
\left\{\begin{array}{cc}-1,&
e\in\vartriangle^-\\
1,&e\in\vartriangle^+\\
0,&\textrm{otherwise}
\end{array}\right.,
\end{equation}
whose examples are given in Fig. \ref{fig-1}.

\vspace{-3mm}

\begin{figure}[htbp]
\begin{minipage}{0.4\linewidth}
\includegraphics[scale=0.8]{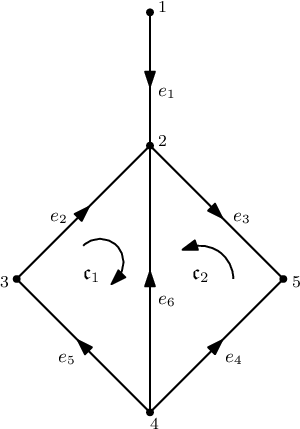}
\end{minipage}
\hspace{-2.5cm}\begin{minipage}{0.6\linewidth}
{\footnotesize\[\mathcal{B}(G)=\left(\begin{array}{ccccc}
-1&1&0&0&0\\
0&1&-1&0&0\\
0&-1&0&0&1\\
0&0&0&-1&1\\
0&0&1&-1&0\\
0&1&0&-1&0
\end{array}\right)
\textrm{and}\;\;
\mathcal{C}(G)=\left(\begin{array}{cccccc}
0&1&0&0&1&-1\\
0&0&-1&1&0&-1
\end{array}\right).\]}
\end{minipage}
\caption{The matrices $\mathcal{B}(G)$ and $\mathcal{C}(G)$ of the presented graph $G$.}
\label{fig-1}
\end{figure}

\newpage

The lemma below indicates the relations among the operators and matrices mentioned above.

\begin{lem}\label{lem-1}
The operator $\operatorname{grad}\operatorname{grad^{\ast}}$ gives the matrix ${\mathcal B}{\mathcal B}^{\top}$ and the operator $\operatorname{curl^{\ast}}\operatorname{curl}$ gives the matrix ${\mathcal C}^{\top}{\mathcal C}$.
\end{lem}

\begin{proof}
Assign $E(G)$ and $T(G)$ arbitrary orientations. For any $v\in V(G)$, let $\delta_v\in
L^2_{\wedge}(V)$ be the function such that $\delta_v(x)=1$ if $x=v$ and $0$ otherwise. For any edge
$e\in E(G)$, let $\delta_{e}\in L^2_{\wedge}(E)$ denote the function such that
$$\delta_e(i,j)=-\delta_e(j,i)=1$$ if $e=\{i,j\}$ and $0$ otherwise. For any
$\vartriangle\in T(G)$, set $\delta_{\vartriangle}\in L^2_{\wedge}(T)$ to be the function such that
$$\delta_{\vartriangle}(i,j,k)=\delta_{\vartriangle}(j,k,i)=\delta_{\vartriangle}(k,i,j)
=-\delta_{\vartriangle}(j,i,k)=-\delta_{\vartriangle}(i,k,j)=-\delta_{\vartriangle}(k,j,i)=1$$
if $\vartriangle= \{i,j,k\}$ and $0$ otherwise. Clearly, $\{\delta_v\mid v\in V(G)\}$,
$\{\delta_e\mid e\in E(G)\}$ and $\{\delta_{\vartriangle}\mid\vartriangle\in T(G)\}$ are
orthonormal basis of $L^2_{\wedge}(V)$, $L^2_{\wedge}(E)$ and $L^2_{\wedge}(T)$, respectively. Assume that ${\rm grad}$
gives the matrix $\mathcal{X}$ and ${\rm curl}$ gives the matrix ${\mathcal{Y}}$ under these basis. It suffices to
show that $\mathcal{X}=\mathcal{B}$ and $\mathcal{Y}={\mathcal C}$ since $f^{\ast}$ gives the matrix $F^*$ if $f$ gives the matrix $F$.

Note that the $(e,v)$-th entry of $\mathcal{X}$ is $$\mathcal{X}_{e,v}=\langle
\operatorname{grad}\delta_v,\delta_e\rangle_E=\sum_{i\le
j}(\operatorname{grad}\delta_v)(i,j)\delta_e(i,j)=\delta_v(e^+)-\delta_v(e^-),$$ which yields that
$\mathcal{X}_{e,v}=1$ if $v=e^+$, $-1$ if $v=e^-$ and $0$ otherwise, and therefore $\mathcal{X}=\mathcal{B}$.

Similarly, the $(\vartriangle,e)$-th entry of $\mathcal{Y}$ is $$\mathcal{Y}_{\vartriangle,e}=\langle
{\rm curl}~\delta_e,\delta_{\vartriangle}\rangle_T=\sum_{x<y<z}({\rm curl}~\delta_e)(x,y,z)
\delta_{\vartriangle}(x,y,z) =\delta_e(i,j)+\delta_e(j,k)+\delta_e(k,i),$$ where
$\vartriangle=(i,j,k)$. Hence, if $e\in \vartriangle^+$ then $e\in\{(i,j),(j,k),(k,i)\}$
and thus $\mathcal{Y}_{\vartriangle,e}=1$; if $e\in\vartriangle^-$ then $e\in\{(j,i),(k,j),(i,k)\}$ and
thus $\mathcal{Y}_{\vartriangle,e}=-1$; if $e\not\in\vartriangle$ then $\mathcal{Y}_{\vartriangle,e}=0$. Thereby,
$\mathcal{Y}=\mathcal{C}$.
\end{proof}

\begin{thm}\label{H-def}
Let $G$ be a graph with orientations on its edge set $E$ and triangle set $T$. The graph Helmholtzian gives the square matrix $\mathcal{H}(G)=(h_{ee'})$ indexed by the edge set
of $G$ with
$$
h_{ee'}=
\begin{cases}
\triangle(e)+2, & \mbox{if $e'=e$};\\
-1, & \mbox{if $e\leftrightarrow e'$ and $e\not\vartriangle e'$};\\
1, & \mbox{if $e'\overset{\pm}{\sim} e$ and $e\not\vartriangle e'$}; \\
0, & \mbox{otherwise}.
\end{cases}
$$
\end{thm}

\begin{proof}
From Lemma \ref{lem-1}, by \eqref{eqhelm} the vector Laplacian $\mathcal{H}(G)$ gives the graph matrix
\begin{equation}\label{H=B+C}
\mathcal{H}(G)={\mathcal B}(G){\mathcal B}(G)^{\top} + {\mathcal C}(G)^{\top}{\mathcal C}(G).
\end{equation}
By immediate calculations, the $(e,e')$-th entry of $\mathcal{H}(G)$ is given as
\[\begin{array}{lll}
\mathcal{H}_{e,e'}&=&(\mathcal{B}\mathcal{B}^{\top})_{e,e'}+(\mathcal{C}^{\top}\mathcal{C})_{e,e'}\\[2mm]
&=&\sum_{v\in V(G)}\mathcal{B}_{e,v}\mathcal{B}^{\top}_{v,e'}+\sum_{\vartriangle\in T(G)}(\mathcal{C}^{\top})_{e,\vartriangle}\mathcal{C}_{\vartriangle,e'}\\[2mm]
&=&\sum_{v\in V(G)}\mathcal{B}_{e,v}\mathcal{B}_{e',v}+\sum_{\vartriangle\in T(G)}\mathcal{C}_{\vartriangle,e}\mathcal{C}_{\vartriangle,e'}\\[2mm]
&=&\sum_{v\in \mathcal{V}(e)\cap \mathcal{V}(e')}\mathcal{B}_{e,v}\mathcal{B}_{e',v}+\sum_{e,e' \in \vartriangle}\mathcal{C}_{\vartriangle,e}\mathcal{C}_{\vartriangle,e'}.
\end{array}\]
If $e=e'$, then $$\sum_{v\in \mathcal{V}(e)\cap
\mathcal{V}(e')}\mathcal{B}_{e,v}\mathcal{B}_{e',v}=\sum_{v\in \mathcal{V}(e)}\mathcal{B}_{e,v}^2=2 \;\; \mbox{and}\;\sum_{e,e' \in \vartriangle
}\mathcal{C}_{\vartriangle,e}\mathcal{C}_{\vartriangle,e'}=\sum_{e \in \vartriangle}\mathcal{C}_{\vartriangle,e}^2=\triangle(e),$$
which results in $\mathcal{H}_{e,e'}=\triangle(e)+2$. If $e\not\sim e'$, then
$$\sum_{v\in \mathcal{V}(e)\cap \mathcal{V}(e')}\mathcal{B}_{e,v}\mathcal{B}_{e',v}=0 \;\; \mbox{and}\;\sum_{e,e' \in \vartriangle}\mathcal{C}_{\vartriangle,e}\mathcal{C}_{\vartriangle,e'}=0,$$
which shows $\mathcal{H}_{e,e'}=0$. It remains to consider the case $e\sim e'$. If $e\vartriangle e'$
and $e\leftrightarrow e'$, say $e,e'\in\vartriangle_0$ and $e\rightarrow e'$, then
$$\sum_{v\in \mathcal{V}(e)\cap \mathcal{V}(e')}\mathcal{B}_{e,v}\mathcal{B}_{e',v}=\mathcal{B}_{e,e^+}\mathcal{B}_{e',{e'}^-}=-1\;\;
\mbox{and}\;\sum_{e,e' \in \vartriangle
}\mathcal{C}_{\vartriangle,e}\mathcal{C}_{\vartriangle,e'}=\mathcal{C}_{\vartriangle_0,e}
\mathcal{C}_{\vartriangle_0,e'}=1$$
 indicating  $\mathcal{H}_{e,e'}=0$. If $e \vartriangle e'$ and $e\overset{\pm}{\sim} e'$, say
$e,e'\in\vartriangle_1$ and $e\overset{+}{\sim}e'$, then
$$\sum_{v\in \mathcal{V}(e)\cap \mathcal{V}(e')}\mathcal{B}_{e,v}\mathcal{B}_{e',v}=\mathcal{B}_{e,e^+}\mathcal{B}_{e',{e'}^+}=1\;\;
\mbox{and}\sum_{e,e' \in \vartriangle
}\mathcal{C}_{\vartriangle,e}\mathcal{C}_{\vartriangle,e'}=\mathcal{C}_{\vartriangle_1,e}
\mathcal{C}_{\vartriangle_1,e'}=-1,$$ which yields
$\mathcal{H}_{e,e'}=0$. If $e\not\vartriangle e'$ and $e \leftrightarrow e'$, say $e\rightarrow e'$,
then $$\sum_{v\in \mathcal{V}(e)\cap \mathcal{V}(e')}\mathcal{B}_{e,v}\mathcal{B}_{e',v}=\mathcal{B}_{e,e^+}\mathcal{B}_{e',{e'}^-}=-1\;\; \mbox{and}\sum_{e,e' \in \vartriangle}\mathcal{C}_{\vartriangle,e}\mathcal{C}_{\vartriangle,e'}=0,$$ which arrives at
$\mathcal{H}_{e,e'}=-1$. If $e\not\vartriangle e'$ and $e\overset{\pm}{\sim} e'$, say
$e\overset{+}{\sim}e'$, then
$$\sum_{v\in \mathcal{V}(e)\cap \mathcal{V}(e')}\mathcal{B}_{e,v}\mathcal{B}_{e',v}=\mathcal{B}_{e,e^+}\mathcal{B}_{e',{e'}^+}=1\;\; \mbox{and}\;
\sum_{e,e' \in \vartriangle}\mathcal{C}_{\vartriangle,e}\mathcal{C}_{\vartriangle,e'}=0$$ implying
$\mathcal{H}_{e,e'}=1$.

The proof is completed.
\end{proof}

\begin{example}\label{ex-0}
In Theorem \ref{H-def}, we determine the graph matrix representation of graph Helmholtzian. For the graph $G$ in Fig. \ref{fig-1}, by Theorem \ref{H-def} one can verify that
\[\mathcal{H}(G)=\left(\begin{array}{cccccc}
2&1&-1&0&0&1\\
1&3&-1&0&0&0\\
-1&-1&3&0&0&0\\
0&0&0&3&1&0\\
0&0&0&1&3&0\\
1&0&0&0&0&4
\end{array}\right),\]
which is just $\mathcal{B}(G)\mathcal{B}(G)^{\top} +\mathcal{C}(G)^{\top}\mathcal{C}(G)$.
\end{example}


\section{Vector Laplace Equation on Graphs}\label{H47}

Clearly, the graph-theoretic analogue of vector Laplace equation \eqref{vector-F} is
\begin{equation}\label{vector-F-G}
\mathcal{H}(G) {\bf x} = {\bf 0}.
\end{equation}
Hereafter, we determine the dimension $\dim{\mathcal{V}_0}$ of solution space $\mathcal{V}_0$ of \eqref{vector-F-G}. Since $\mathcal{H}(G)$ is diagonalizable, then $\dim{\mathcal{V}_0}$ is equal to the algebraic multiplicity of eigenvalues $0$, which is said to be {\it nullity} of a graph $G$. Denoted by  $\eta_M(G)$ the {\it nullity} of a graph $G$ with respect to a graph matrix $M(G)$. For the adjacency matrix $A(G)$, Collatz and Sinogowitz \cite{col-sin} first posed the problem of characterizing all graphs with $\eta_A(G) >0$. This question has strong chemical background, because $\eta(G) = 0$ is a necessary condition for a so-called conjugated molecule to be chemically stable, where $G$ is the graph representing the carbon-atom skeleton of this molecule. For the Laplacian matrix $L(G)$, it is well-known that the nullity of a graph is exactly the number of its connected components.


We next determine $\eta_{\mathcal{H}}(G)$ involving the order, the size and the number of triangles of $G$.

\begin{lem}\label{lem-a-1}
Let $G$ be a graph with size $m(G)$. Then
\[\eta_{\mathcal{H}}(G)=m(G)-{\rm rank}\left(\begin{array}{c}\mathcal{B}(G)^{\top}\\\mathcal{C}(G)\end{array}\right),\]
where $\mathcal{B}(G)$ and $\mathcal{C}(G)$ are defined in \eqref{B-C}.
\end{lem}

\begin{proof}
Given an arbitrary orientation to $E(G)$ and $T(G)$, by \eqref{H=B+C} we get that $$\mathcal{H}(G){\rm \bf x}=(\mathcal{B}(G)\mathcal{B}(G)^{\top}+\mathcal{C}(G)^{\top}\mathcal{C}(G)){\rm \bf x}=\mathbf{0}$$ if and only if $\mathcal{C}(G){\rm \bf x}=\mathbf{0}$ and $\mathcal{B}(G)^{\top}{\rm \bf x}=\mathbf{0}$, which is equivalent to $$\left(\begin{array}{c}B(G)^{\top}\\\mathcal{C}(G)\end{array}\right){\rm \bf x}=\mathbf{0}.$$ Thereby, $\eta_{\mathcal{H}}(G)=m(G)-{\rm rank}\left(\begin{array}{c}\mathcal{B}(G)^{\top}\\\mathcal{C}(G)\end{array}\right)$.
\end{proof}

Note that ${\rm rank}(\mathcal{B}^{\top}(G))={\rm rank}(\mathcal{B}(G))$ and ${\rm rank}(\mathcal{C}(G))=t_G(\triangle)$, the number of triangles in $G$. Thus we get the following result immediately.

\begin{cor}\label{cor-a-1}
Let $G$ be a graph with arbitrary orientations on $E(G)$ and $T(G)$. Then
\[\eta_{\mathcal{H}}(G)=m(G)-t_G(\triangle)-{\rm rank}(\mathcal{B}(G)).\]
\end{cor}

\begin{lem}\label{lem-3}
If $G$ has a pendent vertex $v$ and $G'=G-v$, then $\eta_{\mathcal{H}}(G)=\eta_{\mathcal{H}}(G')$.
\end{lem}

\begin{proof}
Given arbitrary orientations to $E(G)$ and $T(G)$, we get $m(G')=m(G)-1$, ${\rm rank}(\mathcal{B}(G'))={\rm rank}(\mathcal{B}(G))-1$ and $t_{G'}(\triangle)=t_G(\triangle)$. Hence, the result follows from Corollary \ref{cor-a-1}.
\end{proof}

For a vertex $v$ of a graph $G$, let $N_G(v)$ denote the set of the neighbours of $v$. If $e=uv$ such that $N_G(u)\cap N_G(v) = \emptyset$, the {\it contraction} of $e$ is the replacement of $u$ and $v$ with a single vertex whose incident edges are the edges other than $e$ that are incident to $u$ or $v$, and the resulting graph is denoted by $G/e$.

\begin{lem}\label{lem-4}
Let $G$ be a graph with edge $e=uv$ such that $N(u)\cap N(v)=\emptyset$. If $G'=G/e$, then $$\eta_{\mathcal{H}}(G)=\eta_{\mathcal{H}}(G')+\triangle_{G'}(e).$$
\end{lem}

\begin{proof}
Give orientations to $E(G)$ and $T(G)$ such that $u$  and $v$ are respectively the tail and  the head of any edge incident to it. Clearly, $m(G') = m(G)-1$ and $t_{G'}(\triangle)=t_G(\triangle)-\triangle_G(e)$. It suffices to show that ${\rm rank}(\mathcal{B}(G))={\rm rank}(\mathcal{B}(G'))+1$ by Corollary \ref{cor-a-1}. For any vertex $z$, let $E_z=\{e\mid z\in \mathcal{V}(e)\}$. Denote by $E_1=E(G)\setminus(E_u\cup E_v)$, $E_2=E_u\setminus\{e\}$, $E_3=E_v\setminus\{e\}$, and $u'$ the vertex in $G'$ that replaces $u$ and $v$ in $G$. Then
\[
\mathcal{B}(G)^{\top}=\!\!\!\!\!\!\!\begin{array}{cc}
&\begin{array}{cccc}
E_1&E_2&E_3&\{e\}
\end{array}\\
\begin{array}{c}
u\\v\\ V\setminus\{u,v\}
\end{array}&
\left(\begin{array}{cccc}
\mathbf{0}&-\mathbf{1}^{\top}&\mathbf{0}&-1\\
\mathbf{0}&\mathbf{0}&\mathbf{1}^{\top}&1\\
B_1&B_2&B_3&\mathbf{0}
\end{array}\right)
\end{array} \;\; {\mbox{and}} \;\;
\mathcal{B}(G')^{\top}=\!\!\!\!\!\!\begin{array}{cc}
&\begin{array}{ccc}
E_1&E_2&E_3
\end{array}\\
\begin{array}{c}
u'\\ V'\setminus\{u'\}
\end{array}&
\left(\begin{array}{ccc}
\mathbf{0}&-\mathbf{1}^{\top}&\mathbf{1}\\
B_1&B_2&B_3
\end{array}\right)
\end{array}.
\]

If the $u'$-th row of $B(G')^{\top}$ could be represented by the other rows, then so could the summation of the $u$-th row and the $v$-th row of $B(G)^{\top}$, and thus $${\rm rank}(\mathcal{B}(G'))={\rm rank}(B_1\; B_2\; B_3) \;\; {\mbox{and}} \;\; {\rm rank}(\mathcal{B}(G))={\rm rank}(B_1\; B_2\; B_3)+1.$$ Hence, ${\rm rank}(\mathcal{B}(G))= {\rm rank}(\mathcal{B}(G'))+1$. If the $u'$-th row cannot be represented by the other rows, then $${\rm rank}(\mathcal{B}(G'))= {\rm rank}(B_1\; B_2\; B_3)+1\;\; {\mbox{and}} \;\; {\rm rank}(\mathcal{B}(G))= {\rm rank}(B_1\; B_2\; B_3)+2,$$ and we still have ${\rm rank}(\mathcal{B}(G))= {\rm rank}(\mathcal{B}(G'))+1$.
\end{proof}

The following result follows from the above lemma.

\begin{cor}\label{cor-1}
Let $G$ be a graph with a cut-edge $e$. Then $\eta_{\mathcal{H}}(G)=\eta_{\mathcal{H}}(G/e)$.
\end{cor}

The next result follows from the fact that each edge of a tree is a cut-egde and using Corollary \ref{cor-1} repeatedly.

\begin{cor}\label{cor-2}
Let $\mathcal{T}$ be a tree. Then, $\eta_{\mathcal{H}}(\mathcal{T})=0$ and so $\mathcal{H}(\mathcal{T})$ is an invertible matrix.
\end{cor}

Our main result in this subsection is shown as follows.

\begin{thm}\label{thm-1}
Let $G$ be a graph with $\omega(G)$ components. Then $$\eta_{\mathcal{H}}(G)=m(G)-n(G)-t_G(\triangle)+\omega(G).$$
\end{thm}

\begin{proof}
For $1 \leq i \leq \omega(G)$, we consider each component $G_i$ of $G$ by induction on $t_{G_i}(\triangle)$.  If $t_{G_i}(\triangle)=0$, then ${\rm rank}(\mathcal{B}(G_i))={\rm rank}(\mathcal{B}(G_i)^{\top})={\rm rank}(L(G_i))=n_i-1$, because $G_i$ is connected. Therefore, $\eta_{\mathcal{H}}(G_i)=m(G_i)-n(G_i)+1$. Assume that the statement is true for $t_{G_i}(\triangle) \le t$. Assume that $t_{G_i}(\triangle)=t+1$. Taking an edge $e$ belonging to some triangle, let $\tilde{G}_i$ be the graph obtained from $G_i$ by replacing $e$ with an directed path $e^-\rightarrow u\rightarrow e^+$. Clearly, $G_i=\tilde{G_i}/\{e^-,u\}$ and $N_{\tilde{G_i}}(e^-)\cap N_{\tilde{G_i}}(u)=\emptyset$. Along with Lemma \ref{lem-4} we get
\begin{equation}\label{eta1}
\eta_{\mathcal{H}}(\tilde{G}_i)=\eta_{\mathcal{H}}(G_i)+\triangle_{G_i}(e).
\end{equation}
By inductive assumption,
\begin{equation}\label{eta2}
\eta_{\mathcal{H}}(\tilde{G}_i)=m(\tilde{G}_i)-n(\tilde{G}_i)-t(\tilde{G}_i)+1.
\end{equation}
Substituting $m(\tilde{G}_i)=m(G_i)+1$, $n(\tilde{G}_i)=n(G_i)+1$ and $t_{\tilde{G}_i}(\triangle)=t_{G_i}(\triangle)-\triangle_{G_i}(e)$ into \eqref{eta2}, by \eqref{eta1} we arrive at
\[(m(G_i)+1)-(n(G_i)+1)-(t_{G_i}(\triangle)-\triangle_{G_i}(e))+1=\eta_{\mathcal{H}}(\tilde{G}_i)
=\eta_{\mathcal{H}}(G_i)+\Delta_{G_i}(e)\]
resulting in $\eta_{\mathcal{H}}(G_i)=m(G_i)-n(G_i)-t_{G_i}(\triangle)+1$. In consequence,
$$\eta_{\mathcal{H}}(G)=\sum_{i=1}^{\omega(G)}\eta_{\mathcal{H}}(G_i)
=\sum_{i=1}^{\omega(G)}[m(G_i)-n(G_i)-t_{G_i}(\triangle)+1]=m(G)-n(G)-t_G(\triangle)+\omega(G).$$

This completes the proof.
\end{proof}

The following results instantly follows from the above theorem.

\begin{prop}\label{thm-cor}
Under the conditions of Theorem \ref{thm-1},
\begin{itemize}
\item[{\rm (i)}]
$\dim \ker(\Lambda) = \dim \mathcal{V}_{0}  = m(G)-n(G)-t_G(\triangle)+\omega(G)$;
\item[{\rm (ii)}]
$t_G(\triangle) =  m(G)-n(G)-\eta_{\mathcal{H}}(G)+\omega(G)$.
\end{itemize}
\end{prop}

\section{Further remarks}

In the article we have regarded a graph as a discrete analogue of vector Laplacian and vector Laplace equation. In Theorem \ref{H-def},  we give the graph matrix representation $\mathcal{H}(G)$ of vector Laplacian, which is called {\it Helmholtzian matrix} of a graph $G$. Historically, it is the first one that is indexed by the edge set of a graph, compared with all the other graph matrices indexed by the vertex set of a graph. On the other hand, the graph Helmholtzian has been applied in the statistical ranking \cite{jiang-lim} and the random walks on simplicial complexes \cite{sch-Siam}. Just as the roles of the graph (normalized) Laplacian in studying the structural and dynamical properties of ordinary networks \cite{chung-bool1,chung-book}, we look forward to the applications of graph Helmholtzian in the simplicial network, due to the Helmholtz operator is a special case of Hodge Laplacians based upon the clique complex of a graph. Henceforward, a spectral theory based on graph Helmholtzian is expected \cite{lu-shi-wang}.

In Proposition \ref{thm-cor}(i), we get the dimension of solution space of a discrete analogue of PDE: vector Laplacian, which possesses natural counterparts on graphs. Moreover, the $\ker(\Lambda)$ is involved in the Helmholtz Decomposition Theorem, a special case of Hodge decomposition holding in general for any simplicial complex of any dimension \cite[Theorem 2]{jiang-lim}:\smallskip

\noindent{\bf Helmholtz Decomposition Theorem\cite[Theorem 2]{jiang-lim}.}
Let $G$ be a graph and $K_G$ be its clique complex. The space of edge flows on $G$, i.e.
$C^1(K_G,\mathbb{R}) = L^2_\wedge(E)$, admits an orthogonal decomposition
$$C^1(K_G,\mathbb{R}) = {\rm im}(\operatorname{grad}) \oplus  \ker(\Lambda) \oplus {\rm im}(\operatorname{curl}^*).$$
See \cite[Section 6.3]{lim2020} for more applications of the Hodge Laplacian and Hodge decomposition on graphs to other fields.

In the end, the number of triangles in a graph/network is a main metric to extract insights for an extensive  range of graph applications, see \cite[eg.]{has-dav,pan-etal,rez-etal,zhang-etal} for more details. Hence, the significance of triangle counting is posed by the GraphChallenge competition \cite{kep}, which is now known in Proposition \ref{thm-cor}(ii).

\section*{Acknowledgments}

Jianfeng Wang expresses his sincere thanks to Prof. Lek-Heng Lim for his kind suggestion. Jianfeng Wang is  supported by National Natural Science Foundation of China (No. 12371353) and Special Fund for Taishan Scholars Project. Lu Lu is  supported by National Natural Science Foundation of China (No. 11671344).


PS: In their longer manuscript \cite{lu-shi-wang}, the last two authors of this paper,Yongtang Shi and Yi Wang have considered the spectral properties of graph Helmholtzian, including  the irreducibility, the interlacing theorem, the graphs with few Helmholtzian eigenvalues, the coefficients of Helmholtzian polynomial, the relations between Helmholtzian and Laplacian spectra, the Helmholtzian spectral radii and their limit points, the least Helmholtzian eigenvalue, the product graphs and the Helmholtzian integral graphs. On the other hand, Jianfeng Wang and his student Zhen Chen have determined the Helmholtzian eigenvalues of threshold graphs recently.
\end{document}